\documentclass[a4paper, 12pt]{article}
\usepackage{graphics}
\usepackage{amsmath,amsfonts,amstext,amscd,amssymb,a4,enumerate}
\usepackage[latin1]{inputenc}
\usepackage{graphicx}
	
\usepackage{amsthm}

\usepackage{color} 

\everymath={\displaystyle}

\newcommand{\cM}{{\mathcal{M}}}
\newcommand{\cO}{{\mathcal{O}}}

\newcommand{\cC}{{\mathcal{C}}}
\newcommand{\cD}{{\mathcal{D}}}
\newcommand{\cH}{{\mathcal{H}}}

\newcommand{\cS}{{\mathcal{S}}}

\newcommand{\cU}{{\mathcal{U}}}

\newcommand{\RR}{{\mathbb{R}}}

\newcommand{\fkl}{{\mathfrak{lo}}}
\newcommand{\fklzero}{{\mathfrak{lo}^0}}
\newcommand{\fku}{{\mathfrak{up}}}
\newcommand{\fko}{{\mathfrak{o}}}

\newcommand{\p}{\mathrm{p}}

\newcommand{\bGL}{{\bf GL}}
\newcommand{\bGLpos}{{\bf GL^{\pos}}}
\newcommand{\bGLp}{{\bf GL^+}}
\newcommand{\bL}{{\bf Lo}}
\newcommand{\bLp}{{\bf Lo^+}}

\newcommand{\bO}{{\bf O}}
\newcommand{\bOpos}{{\bf O^{\pos}}}
\newcommand{\bSO}{{\bf SO}}
\newcommand{\bLone}{{\bf Lo^1}}
\newcommand{\bU}{{\bf Up}}
\newcommand{\bUp}{{\bf Up^+}}

\newtheorem{theo} {Theorem}
\newtheorem{lemma} {Lemma}[section]
\newtheorem{prop}[lemma] {Proposition}
\newtheorem{coro}[lemma] {Corollary}

\theoremstyle{remark}

\newtheorem{remark}[lemma]{Remark}

\newcommand{\e}{\, = \, }

\newcommand{\diag}{\operatorname{diag}}

\newcommand{\pos}{\operatorname{pos}}

\newcommand{\sgn}{\operatorname{sgn}}

\newcommand{\svd}{\operatorname{svd}}

\date{}
\begin{document}

	\title{Linearizing  Toda and SVD flows  on \\ large phase spaces of matrices with real spectrum}
	\author{Ricardo S. Leite, Nicolau C. Saldanha, \\ David Mart\'inez Torres and Carlos Tomei }


	\maketitle	
	
	\centerline{\it To the memory of Hermann Flaschka}

\begin{abstract}
We consider different phase spaces for the Toda flows  (\cite{Flaschka,Moser}) and the less familiar SVD flows (\cite{Chu, Li}). For the Toda flow, we handle symmetric and non-symmetric matrices with real simple eigenvalues, possibly with a given profile. Profiles encode, for example, band matrices and Hessenberg matrices. For the SVD flow, we assume simplicity of the singular values.
In all cases, an open cover is constructed, as are corresponding charts to Euclidean space. The charts linearize the flows, converting it into a linear differential system with constant coefficients and diagonal matrix. A variant construction transform the flows into uniform straight line motion.  Since limit points belong to the phase space, asymptotic behavior becomes a local issue.  The constructions rely only on basic facts of linear algebra, making no use of symplectic geometry.
\end{abstract}

\medbreak


\smallbreak

{\noindent\bf MSC-class:}  65F18; 15A29; 37J35.

\section{Introduction}

The Toda flow  -- a Hamiltonian system of $n$ particles in the line --  was famously introduced by Hermann Flaschka \cite{Flaschka} to the integrable community   by casting it as a Lax pair inducing a flow $J(t)$ on Jacobi matrices,
	\begin{equation} \label{TodaJacobi}
	\dot J \e [ J , \Pi_\fko \,  J] \e [ J , - \Pi_\fku \,  J]		 , \quad J(0) = J_0 \, .
	\end{equation}
Here, $\Pi_\fko$ and $\Pi_\fku$ are projections associated with the direct sum $\cM \e \fko \oplus \fku$, where $\cM$ is the vector space of real, $n \times n $ matrices, $\fko$ and $\fku$ are the Lie algebras of skew symmetric and upper triangular matrices respectively.

The standard symplectic structure in $\RR^{2n}$  yields a symplectic structure on Jacobi matrices of trace zero, after getting rid of the simple dynamics of the center of mass.
	The  eigenvalues of $J(t)$ stay fixed along the flow, and Poisson commute: they  provide the required conserved quantities to yield the complete integrability of the system.
	
	The subject expanded substantially. 
	Here we list a minimal set of subsequent advances, leading to the issues we consider in this text;
	see \cite{DLSTT} for more detailed historical information. Moser \cite{Moser} obtained angle variables, essentially the discrete counterpart of the norming constants used as inverse scattering variables for Schrödinger operators. Adler \cite{Adler} showed that Jacobi matrices of trace zero  are a coadjoint orbit of the group of upper triangular matrices. Deift et al. (\cite{DLNT1, DLT2}) then showed integrability of the Toda flow on larger coadjoint orbits, first of generic real, symmetric matrices, then generic real, non-symmetric matrices. 
	
	To fix notation, given a real polynomial $p$, the  {\it Toda flow} is the solution of the differential equation (derivatives with respect to time are denoted by dots)
\begin{equation} \label{Toda}
	\dot M \e [ M , \Pi_\fko \,  p(M)] , \quad M(0) = M_0 \, ,
\end{equation}
where $p(M)$ is the evaluation of the polynomial $p(x)$ at $x = M$.
	
	\medskip
	Here we introduce {\it linearizing variables} on real matrices, symmetric or not, with real, simple spectrum. A first example were {\it bidiagonal variables} for Jacobi matrices \cite{LST1}. An extension  \cite{TT} considered  the manifold of full flags of any non-compact real semisimple Lie algebra and its Hessenberg-type submanifolds. These references and the present self-contained paper make no use of symplectic theory, the constructions using classical tools from linear algebra.

\bigskip	
Let $\bGL$ denote the group of invertible matrices.
 The orthogonal group and its Lie algebra are denoted by $\bO$ and $\fko$. The (disconnected) groups of invertible lower and upper triangular matrices are, respectively, $\bL$ and $\bU$. The connected component of the identity are the subgroups $\bLp$ and $\bUp$.  The  nilpotent group  of lower triangular matrices with diagonal entries equal to 1 is $\bLone$ with Lie algebra $\fkl^0$. 
	
\bigskip
We present the linearization of Toda the flows in minimal form; for the basic linear algebra vocabulary, see the Appendix. Let $M = X^{-1} D X$ be a real, diagonalizable matrix with real simple spectrum, where the entries of $ D = \diag(\lambda_1, \ldots, \lambda_n)$, 
are in decreasing order. Generically, the determinants of the $k \times k$ top submatrices of $X$ are nonzero: suppose they are positive.  From the $QR$ decomposition $X = \tilde L Q$ ($\tilde L \in \bLp$  and $Q \in \bO$), 
\[ M \e X^{-1} D X \e Q^\top \tilde L^{-1} D \tilde L Q \e Q^\top (Y + D) Q,   \qquad Y \in \fkl^0 . \]
 Again, from the $LU$ decomposition $Q = LU$ ($ L \in \bLone$  and $U \in \bU$),
\begin{equation} \label{resumo}
	 M \e Q^\top (Y + D) Q \e U^{-1} L^{-1} (Y + D) L U \e U^{-1} (Z + D) U, \qquad Z \in \fkl^0 .
	 \end{equation}
The triple of matrices $(D, Y, Z)$ recovers $M$ and, if $M$ evolves with a Toda flow,
\[ \dot D^\pi = \dot Y =  0 \, ,  \qquad \dot Z = [Z, -p(D)] \,  .\]
Thus, the entries of $Z$ evolve according to the ODE's $\dot z_{ij} = (p(\lambda_i) - p(\lambda_j)) \, z_{ij}$: the system decouples. In a nutshell, this is the content of Section \ref{Euclideancharts}.	
	
\bigskip	
Let $D$ be a real diagonal matrix with simple spectrum.	The symmetric and non-symmetric  {\it isospectral manifolds} are the orbits
	\begin{equation}
		\cO^\bO_D \e \{ Q^\top \, D \, Q , Q \in \bO\} \, , \quad \cO_D^\bGL \e \{ X^{-1} \, D \,X, X \in \bGL\} \, .
	\end{equation}
Let $S_n$ be the symmetric group of permutations of $\{1, 2, \ldots,n\}$.  Write $\cO$ to denote either $\cO^\bO_D$ or $\cO_D^\bGL$.

	\begin{theo} \label{main} Each isospectral manifold  admits an atlas  with charts  $\phi^\pi: \cU^\pi_D \subset \cO \to \RR^N, \pi \in S_n, N = \dim \cO$. The following properties hold.
	
	\begin{enumerate}
		\item 
		Each chart domain $\cU^\pi_D$ is an open dense set of $\cO$. 
		Each diagonal matrix in $\cO$ belongs to a unique $\cU^\pi_D$. Each chart $\phi^\pi$ is a diffeomorphism.
	\item Each set $\cU^\pi_D$ is invariant under the Toda flows.
	    In these variables (i.e., in the images of the charts), the Toda flow associated with a polynomial $p$ is 
	    \[(\dot y,\dot z) \e (0, C z), \quad y \in\RR^{N_y}, \quad z \in \RR^{N_z}, \quad N = N_y + N_z \, , \]  for a constant diagonal matrix $C$ (dependent on $p$). If $\cO = \cO_D^\bO$, $N_y = 0$.
	
	    \end{enumerate}
    \end{theo}
Theorem \ref{main} is the juxtaposition of results obtained in Section \ref{Euclideancharts}.
The results are extended in Section  \ref{profiles} to matrices of a fixed profile, to be defined later. Charts restrict well to subspaces  of banded matrices, symmetric or not.

\bigskip
The connection of Toda flows with physics and numerical analysis led to the study of asymptotic behavior of such flows (\cite{Symes, DNT}). In the symmetric case, limit points of the differential equations are diagonal matrices, in the non-symmetric case they are upper triangular. Since these limit points belong to some chart domain $\cU^\pi_D$, convergence now is a matter of local theory. Bidiagonal variables on Jacobi matrices were  employed in \cite{LST1} to compute the Toda phase shift originally obtained by Moser \cite{Moser} and  P. Deift (personal communication). In \cite{LST2}, they were used  to provide a complete description of the dynamics of the classical Wilkinson's shift strategy to compute eigenvalues of Jacobi matrices \cite{Parlett}.

\bigskip
Chu \cite{Chu} and Li \cite{Li} considered the so called SVD flows,
\begin{equation}
	\dot M = M  \, \big( \Pi_\fko \, p(M^\top M) \big) - \big(\Pi_\fko \, p(M M^\top)\big) \, M ,
\end{equation}
which have the property of preserving the singular values of the initial condition. Chu related the flows to numerical algorithms and computed asymptotics of solutions. Li obtained integrability on an appropriate symplectic space of large dimension. Our techniques apply also to this context, as we show in Section \ref{SVDlinearized}.

\medskip
For $\Sigma= \diag(\mu_1, \ldots, \mu_n), \mu_1 > \mu_2 > \ldots > \mu_n > 0$, the set
\[ \cO^{\svd}_{\Sigma} \e  \{ Q^\top \, \Sigma\, U , \  Q, U \in \bO \} \ \]
consists of the real $n \times n $ matrices with singular values equal to $\Sigma$.

Let $\cD^{\sgn} \subset \bGL$ be the group of sign diagonal matrices.

\begin{theo} \label{mainSVD} The set  $\cO^{\svd}_{\Sigma}$ admits an atlas  with charts indexed by $(\pi,\rho,E) \in S_n \times S_n \times \cD^{\sgn}$, $\phi_{\pi,\rho,E}: \cU_{\pi,\rho,E} \subset \cO^{\svd}_{\Sigma} \to \RR^N, N = \dim \cO^{\svd}_{\Sigma}$.
	
	\begin{enumerate}
		\item 
		Each  $\cU_{\pi,\rho,E}$ is an open  subset of $\cO^{\svd}_{\Sigma} $. 
		Each chart is a diffeomorphism.
		\item The sets $\cU_{\pi,\rho,E}$ are invariant under the Toda flows.
		In these variables, the SVD flow associated with a polynomial $p$ is $\dot v \e C v \in \RR^N$ for a constant diagonal matrix $C$ (dependent on $p$).
		
	\end{enumerate}
\end{theo}

\bigskip

   Once the charts are constructed for Toda flows, linearization (Subsection \ref{Todalinearized}) and then conversion  to uniform straight line motion (Subsection \ref{lines}) are easy. Analogous results hold for SVD flows (Section \ref{SVDlinearized}).  In a slight amplification of Theorem \ref{main}, an enlarged change of variables incorporates possible changes of eigenvalues. Integrability (or stronger, superintegrability, as in \cite{Agrotis}) then is a matter of restricting parts of Euclidean space to appropriate subdomains, as exemplified in Section \ref{example} for $3 \times 3$ Jacobi matrices, and then proved in Subsection \ref{lines}. 

\medskip
   In Section \ref{example}, some results are presented for matrices of dimension $3$. Sections \ref{Euclideancharts} and \ref{SVDlinearized} contain the proofs of the theorems above. We also present some geometric results: Toda flows are just uniform straight line motion in Euclidean space and, for monotonic polynomials $p$, are superintegrable in the sense of \cite{Agrotis}. In Section \ref{calculi}, we describe new isospectral flows easily solved in linearizing coordinates which additionally are implemented at discrete times by $QR$-type steps.
   
   It would be interesting to study the case of complex or multiple eigenvalues.

    \bigskip

    \noindent{\bf Acknowledgements}  Tomei thanks Luen-Chau Li for many interesting messages. Saldanha gratefully acknowledges support
    from CNPq (306756/2021-8), and Tomei from CNPq (304742/2021-0) and FAPERJ (E32/2021CNE).

    \section{Notation and $3 \times 3$ examples} \label{example}
    
    \subsection{Notation}
    
    All matrices in this text are real. The matrix $M^\top$ is the transpose  of $M$. The notations $\bGL, \bL, \bU, \bO$ denote respectively the (disconnected) groups of real $n \times n$ invertible, lower, upper triangular and orthogonal matrices.  Their respective connected components of the identity are $\bGLp, \bLp, \bUp$ and $\bSO$. Also,  $\bLone \subset \bLp$ is the subgroup of matrices with diagonal entries equal to one. 
 
    Let $\bGLpos$ and $\bOpos$ denote the subsets of $\bGL$ and $\bO$ of matrices with positive principal minors (minors are the determinants of the sub-blocks obtained from a matrix by considering the entries in the first $k$ rows and columns, $k= 1, \ldots, n$). 
    
    Similarly, $\cM$  is the vector space of real $n \times n $ matrices,  $\cS \subset \cM$ the subspace of symmetric matrices,  $\fkl, \fkl^0, \fku$ and $\fko$ the Lie algebras of $\bL, \bLone, \bU$ and $\bO$.

    The symmetric group of all permutations of $\{1, 2, \ldots, n\}$ is $S_n$. 
    We denote the identity permutation by $\pi_0 \in S_n$.
    Given a permutation $\pi \in S_n$, let $P_\pi$ be the permutation matrix for which $P_\pi e_i = e_{\pi(i)}$, where $(e_i, i=1, \ldots,n)$ is the canonical basis of $\RR^n$. Notice that $P_{\pi_1 \circ \pi_2} = P_{\pi_1} P_{\pi_2}$.
    
     Diagonal matrices form the vector space $\cD \subset \cM$.  Given a diagonal matrix $D = \diag(d_1, d_2,\ldots, d_n)$,  define \[ D^\pi =  P_{\pi}^\top D P_\pi = \diag(d_{\pi(1)},d_{\pi(2)}, \ldots, d_{\pi(n)}), \quad \pi \in S_n \, . \]
   The open, convex subset $\cD^{\pi_0}$ consists of matrices with diagonal entries in strictly descending order. More generally,  set $\cD^\pi =  P_{\pi}^\top \cD^{\pi_0} P_\pi $. 
      
      The group of diagonal matrices with entries equal to $\pm 1$ is $\cD^{\sgn}$.

\subsection{$3 \times 3$ Jacobi and Hessenberg matrices} \label{3x3}
     We illustrate the results in this text first in the familiar context of Jacobi $3 \times 3$ matrices, and then for $3 \times 3$ upper Hessenberg matrices with real, simple spectrum.  A Jacobi matrix $J$ is a real, symmetric tridiagonal matrix for which $J_{i+1, i} = J_{i, i+1} >0$ for $i=1, \ldots, n-1$. A $3 \times 3$ real matrix $H$ is Hessenberg if $h_{31}=0$.

 Jacobi matrices have simple eigenvalues and eigenvectors have nonzero first and last coordinates \cite{Moser}. 
 Let $\RR^3_{\pi_0} \subset \RR^3$ consist of vectors with strictly decreasing entries and $ {\mathbb{S}}^2_+$ be the (strictly) positive quadrant of the unit sphere in $\RR^3$.  Moser considered the diffeomorphism
    \[ \phi:  J \in \{ \hbox{Jacobi matrices of trace zero} \} \to ((\lambda_1, \lambda_2, \lambda_3), \  c ) \in  \RR^3_{\pi_0} \times {\mathbb{S}}^2_+ \, , \]
where $\lambda_1 > \lambda_2 > \lambda_3$ are the eigenvalues of $J$ and $c = (c_1, c_2, c_3) \in {\mathbb{S}}^2_+$ is a normal vector of positive  first coordinates of the respective unit eigenvectors.

\medskip
An alternative argument \cite{DNT} recovers $J$ from $D = \diag(\lambda_1, \lambda_2, \lambda_3) \in \cD^{\pi_0}$ and $c = (c_1, c_2, c_3)^\top$. First decompose   the $3 \times 3$ matrix $V$ below into a product of an orthogonal and an upper triangular matrix,
\[V = \begin{pmatrix} c & D c & D^2 c \end{pmatrix} = Q  R , \quad Q \in \bO , \quad R \in \bUp \, .\]
Set $ J = Q^\top \Lambda Q$: it turns out that the entries $(i, i+1)$ are automatically positive.

\medskip
	Of fundamental importance in what follows is the fact that the principal minors of $Q $ (and hence $Q^\top$) are nonzero. Indeed, up to multiplication by (nonzero) diagonal entries of $R$, such minors equal the minors of $V$, which are easily seen to be nonzero by computing a Vandermonde determinant and using $c_i \ne 0 $. As the columns of $Q^\top$ are eigenvectors of $J$, each may be multiplied by $-1$ yielding another spectral diagonalization of $J$. Instead of using Moser's normalization $c_i >0$, we  require  $Q^\top \in \bOpos$ (and thus $Q \in \bOpos$).
	
	We now present the variables used in this text, introduced for tridiagonal matrices in \cite{LST1}.
	 Let $\cO$ be the vector space of real, symmetric, tridiagonal $3 \times 3$ matrices.
	Suppose $T \in \cO$ for which $T = Q^\top D Q$, where $Q \in \bOpos$, $D \in \cD^{\pi_0}$.
	The (unique) LU-decomposition of $Q \in \bOpos$ is   $Q = L U $ for $L \in \bLone, U \in \bUp$ (Proposition \ref{fatoracoes}). We then have
    \[ T = U^{-1} L^{-1} D L U \, , \]
   so that $B = L^{-1} D L = U T U^{-1}$ is lower triangular, form the first equality, and upper Hessenberg, from the second. Thus $B$ is lower bidiagonal, and $\diag B = D$.

   From Section \ref{profiles},   the {\it bidiagonal variables} (see \cite{LST1}) induce a diffeomorphism 
    \[ \tilde \phi_e : \cO \cap \{Q^\top D Q, Q \in \bOpos , D \in \cD^{\pi_0}\} \ \to \RR^3_{\pi_0} \times \RR^2 \ , \quad T \mapsto B \ .\] Here, we identify the matrix $B$ with a vector with diagonal entries, in $\RR^3_{\pi_0}$, and a second vector with the two nontrivial off-diagonal entries $b_{21}$ and $b_{32}$. Moreover, Jacobi matrices correspond to matrices  for which $b_{21}, b_{32} > 0 $.

\medskip

   Different permutations of diagonal entries lead to diagonal matrices \[ D^\pi =  P_{\pi}^\top D P_\pi = \diag(\lambda_{\pi(1)},\lambda_{\pi(2)},\lambda_{\pi(3)}), \quad \pi \in S_3 \, , \]
   where $S_3$ is the symmetric group. From Corollary \ref{charts-profiles}, the set of matrices in $\cO$  with simple spectrum  is a manifold covered by  charts indexed by permutations,
   \[ \tilde \phi_{\pi} :  \cO \cap \{ Q^\top D^\pi Q, Q \in \bOpos, D \in \cD^{\pi_0},   D^\pi =  P_\pi^\top D P_\pi\}  \to  \RR^3_\pi \times \RR^2 .  \]
   Here 
   $\RR^3_\pi = \{ (x_{\pi(1)}, x_{\pi(2)},x_{\pi(3)}) \ | \ (x_1, x_2, x_3) \in \RR^3_{\pi_0} \}$.
   Also, $B^\pi = \tilde \phi_{\pi}(T)$ is obtained from $T$ as  in the case $\pi = e$, the trivial permutation:
   \[ \tilde \phi_{\pi}(U^{-1} L^{-1} D^\pi L U ) = L^{-1} D^{\pi} L = B^\pi \, .\]
   
   Similarly, the {\it isospectral set}  of symmetric matrices with a tridiagonal profile with spectrum equal to the spectrum of a fixed matrix $D$ of simple spectrum admit charts
    \[  \phi_{\pi} :  \cO \cap \{ Q^\top D^\pi Q, Q \in \bOpos \}  \to   \RR^2 ,  \]
    obtained by dropping the first coordinates of the map  $\tilde \phi_{\pi}$.
   Each chart is a dense set in the tridiagonal, symmetric, isospectral manifold. Figure 	\ref{ilailah} represents two chart domains, for $\pi_0$ and $\pi_1 = (23)$ in cycle notation, for spectrum $\{ 7, 5, 4\}$. In both cases, the domain is the interior of the polygon with 16 edges. In each case, there exists a continuous map from the compact polygon, including the boundary, to the isospectral manifold; this map is smooth in the interior of the polygon and in the interior of the edges.   
   The full isospectral manifold of tridiagonal symmetric matrices, a torus with two holes (\cite{LST1, Tomei2}), is obtained by identifying segments along the boundary of either polygon. Vertex labels correspond to diagonal matrices: the only vertex in either chart domain is at the center, $\diag(7,5,4)$ and $\diag(7,4,5)$ respectively.  Interior horizontal and vertical segments correspond to matrices with some subdiagonal entry equal to zero. Thus, in the second example, the horizontal segment joining the center $(7,4,5)$ and $(4,7,5)$ correspond to matrices $T$ for which $y_{33}=5, y_{32} = y_{23} = 0$. The four quadrants consist of matrices with constant signs along the subdiagonal entries. The quadrant  $++$ contains all Jacobi matrices. More generally, matrices with nonzero subdiagonal entries belong to all charts domains. 
   The domain of Moser's chart, instead, is the positive quadrant of either chart domain.

   \begin{figure}[ht]
   		\includegraphics[width=0.9\textwidth]{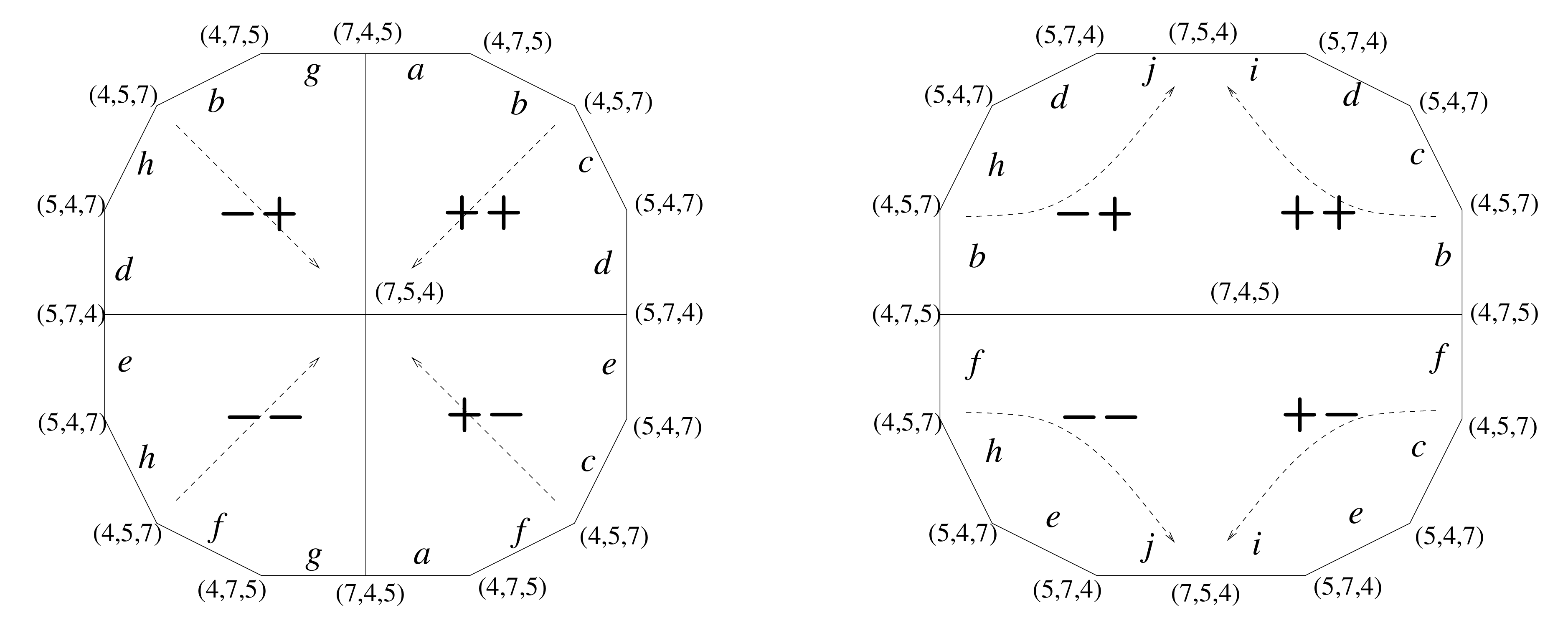}
   	\caption{Jacobi matrices and more }
   	\label{ilailah}
   \end{figure}

Toda flows keep invariant the interior of each quadrant, each edge, each vertex.
As is well known, as $t \to \infty$, standard Toda (the choice $p(x) = x$
in Equation~(\ref{Toda})) starting with a Jacobi matrix
converges to $(7,5,4)$.
Some orbits are represented as dotted arrows in the pictures.
This limit point does not belong to the domain of Moser's chart,
defined on symplectic manifolds,
for which Hamiltonian flows cannot have attractors.
It is however the center of the chart domain associated with $\pi_0$,
and does not belong to the chart domain associated with $\pi_1$.
Our charts provide a fine description of
the asymptotic behavior of numerical algorithms
involving shift strategies \cite{LST2}.

\medskip
The evolution of bidiagonal variables under the standard Toda flow (here, to simplify matters, $\pi = e$, the trivial permutation) is given in Theorem~\ref{linearToda},
\[ \lambda_i'(t) = 0 \, , \quad b_{21}'(t) =  (\lambda_2 - \lambda_1)\, b_{21}(t)\, , \quad  b_{32}'(t) =  (\lambda_3 - \lambda_2)\, b_{32}(t) \, , \]
from which diagonal convergence in the future is immediate, as is the fact that the equations for $b_{21}$ and $b_{32}$ are linear and decoupled. Other geometric properties will be proved in Subsection \ref{lines} from the simplicity of such formulas.

\medskip
We now consider  the vector space $\cH$ of Hessenberg matrices, by
   extending our variables (Proposition \ref{charts}) to matrices in $\cH$  with real, simple spectrum. Linearizing variables consist of {\it pairs} of strictly lower triangular matrices, the second one preserving the Hessenberg {\it profile}, defined in Subsection \ref{profiles},
\[  \tilde \phi_{\pi}^{H}(X^{-1} \, D^\pi \, X )= \left( D^\pi,  \begin{pmatrix} 0 & 0 & 0 \\ y_{21} & 0 &0 \\ y_{31} & y_{32} & 0 \end{pmatrix}, \begin{pmatrix} 0 & 0 & 0 \\ z_{21} & 0 & 0 \\ 0 & z_{32} & 0 \end{pmatrix}    \right) , \quad X \in \bGLpos .   \]
The matrices $Y$ and $Z$ follow  the notation of Equation~(\ref{resumo}). Again from Theorem~\ref{linearToda}, the evolution under the Toda flow is
\[ \lambda_i'(t) = 0 \, , \quad y_{ij}'(t) = 0 \,  \quad z_{21}'(t) =  (\lambda_2 - \lambda_1)\, z_{21}(t)\, , \quad  z_{32}'(t) =  (\lambda_3 - \lambda_2)\, z_{32}(t) \, . \]
For symmetric tridiagonal matrices, $y_{ij}=0$, $z_{21} = b_{21}$ and $ z_{32}= b_{32}$.


    \section{Linearizing Toda flows} \label{Euclideancharts}

    \medskip

     For an open subset $U$ of a real manifold $M$ of dimension $n$, an {\it Euclidean chart} is a diffeomorphism $\phi:U \to \RR^n$.

\subsection{Euclidean charts for isospectral manifolds}

In this subsection  we fix
$ D = \diag(\lambda_1, \ldots, \lambda_n) \in \cD^{\pi_0}$.
Define symmetric and general  {\it isospectral sets},
 \begin{equation} \label{isospectral}
\cO^\bO_D \e \{ Q^\top \,  D\, Q , Q \in \bO\}  \subset \cS\, , \quad \cO_D^\bGL \e \{ X^{-1} \, D \,X, X \in \bGL\} \subset \cM \, .
\end{equation}

The isospectral sets $\cO^\bO_D$ and $\cO_D^\bGL$ are smooth manifolds and they are orbits for the adjoint action of $\bO$ and $\bGL$, respectively.
We define atlases with Euclidean charts  linearizing the Toda flows and some generalizations.

Recall that $D^\pi = P_\pi^\top D P_\pi = \diag(\lambda_{\pi(1)}, \ldots, \lambda_{\pi(n)})$ and set
\begin{equation}\label{cartas}
	\begin{aligned}
 \cU_{D^\pi}^{\bO} &\e \{ Q^\top \ D^\pi \  Q,  Q \in \bOpos \}  \subset \cO^\bO_D,\\ 
\cU_{D^\pi}^\bGL &\e \{ X^{-1} \ D^\pi \ X,  X \in \bGLpos \} \subset \cO_D^\bGL .
\end{aligned} \end{equation}

Let $M \in \cU_{D^\pi}^\bGL$. From  Proposition \ref{fatoracoes} (Schur),
\begin{equation} \label{YpD}
	M = Q^\top (Y +D^\pi) Q \, , \quad Q \in \bOpos, \quad Y \in \fkl^0 \, .
	\end{equation}
Consider the LU factorization $Q = LU$, for $  L \in \bLone , U \in \bUp $, and write
\begin{equation} \label{ZpD}
L^{-1}(Y+ D^\pi) L  = 	Z + D^\pi  \, , \quad Z \in \fkl^0 , 
\end{equation}
so that $M = U^{-1} (Z + D^\pi) U$.

 \begin{prop}  [Euclidean charts for $\cO^\bO_D$ and $\cO_D^\bGL$] \label{charts} The subsets $\cU_{D^\pi}^{\bO} \subset \cO^\bO_D$ and $\cU_{D^\pi}^\bGL \subset \cO_D^\bGL$ are open and dense. They cover $\cO^\bO_D $ and $\cO_D^\bGL$:
 \[ \cO^\bO_D \e \bigcup_{\pi \in S_n} \, \cU_{D^\pi}^{\bO} \ , \quad \cO_D^\bGL \e \bigcup_{\pi \in S_n} \, \cU_{D^\pi}^\bGL \, . \]	
 	For $Y$ and $Z$ as in Equations \eqref{YpD} and \eqref{ZpD}, the maps 
	\[ \phi_{D^\pi}^{\bO} : \cU_{D^\pi}^{\bO} \to \fkl^0 \, , \ \  S \mapsto Z,   \quad \quad \phi_{D^\pi}^\bGL : \cU_{D^\pi}^\bGL \to \fkl^0 \times \fkl^0 \, , \ \  M \mapsto ( Y, Z)   \]
	are diffeomorphisms. 
A matrix $M \in \cU^\bGL_\pi$ is symmetric  (resp. upper triangular) if and only if $Y=0$ (resp $Z=0$).
\end{prop}

The charts provide an atlas for the smooth manifolds $\cO^\bO_D$ and $\cO_D^\bGL$.
We refer to the values of the charts, such as $Y$ and $Z$ or their entries, as  {\it variables}. In Subsection \ref{Todalinearized}, we  obtain formulas for  flows in $\cO_D^\bO$ and $\cO_D^\bGL$  in $(Y,Z)$ variables.

\begin{proof} In light of Equation \eqref{cartas} and the openness of the subsets $\bOpos \subset \bO$ and $\bGLpos \subset \bGL$, chart domains are open. Density is a consequence of the density of $\cD^{\sgn} \bOpos  = \{ EQ , E \in \cD^{\sgn} , Q \in \bOpos\} \subset \bO$ and $\cD^{\sgn} \bGLpos  \subset \bGL$. The fact that the charts domain cover the respective manifolds follow from the (PLU) decomposition of a matrix (see  the Appendix).
	
	 We prove the smooth invertibility of the charts. Let $D^\pi, Y$ and $Z$ as above: we recover $M$ (the recovery of $S$ is the case $Y=0$). We first show that there is a unique $  L \in \bLone$ for which $Z + D^\pi = L^{-1}(Y+ D^\pi) L$. Write $L(Z + D^\pi) = (Y+ D^\pi) L$ and use the simplicity of the spectrum of $D^\pi$: equating entries $(2,1)$, $(3,2), \ldots,(n,n-1)$ in this equation, one obtains the $n-1$ entries  $L_{(k, k-1)}$. Repeat the process for the subsequent subdiagonals
	 until $L$ is computed.
Now factor $L = Q R,$ for $Q \in \bOpos$, $R \in \bUp$, and thus $Q = LU$ for $U = R^{-1} \in \bUp$. Set $M = Q^\top (Y + D^\pi) Q = U^{-1} (Z + D^\pi) U$, as in Equations \eqref{YpD} and \eqref{ZpD}.

Equation $M = Q^\top (Y + D^\pi) Q $  implies that $M$ is symmetric  if and only if $Y=0$. Similarly,  $M = U^{-1} (Z + D^\pi) U$ is upper triangular if and only if $Z = 0$.
\end{proof}

\begin{remark}
The factorization $M = U^{-1} (Z + D) U,\  U \in \bU, Z \in \fkl^0$ is not unique, and hence it does not retrieve $Z$:  for any invertible diagonal matrix $E$,
\[ M = (U^{-1} E^{-1}) (E  ( Z + D) E^{-1}) (E U) = (U^{-1} E^{-1}) ( E Z E^{-1} + D) (E U) \, . \]
\end{remark}

\medskip
 It is convenient to specify the dependency of a matrix on its eigenvalues. Consider the open sets of matrices with simple, real spectrum:
\begin{equation}\label{grandesconjuntos} \cO^\bO = \bigcup_{ D  \in \cD^{\pi_0}} \cO^\bO_{D }  \subset \cS \, , \qquad \cO^\bGL = \bigcup_{ D  \in \cD^{\pi_0}}    \cO^\bGL_D \subset \cM\, .
	\end{equation}
The first subset is dense, while the second is not: matrices with simple spectrum and complex eigenvalues belong to the interior of the complement. We construct charts in  $\cO^\bO$ and $\cO^\bGL$: the chart domains 
 \[ \cU^\bO_\pi \e \bigcup_{D \in \cD^{\pi_0}}  \cU_{D^\pi}^{\bO}  \subset \cO^\bO \; , 
 \quad \cU^\bGL_\pi \e \bigcup_{D \in \cD^{\pi_0}} \, \cU_{D^\pi}^{\bGL}  \subset \cO^\bGL\]
 clearly cover the sets
 \[
 \cO^\bO \e \bigcup_{\pi \in S_n} \, \cU_{\pi}^{\bO} \ , 
 \quad \cO^\bGL \e \bigcup_{\pi \in S_n} \, \cU_{\pi}^\bGL \, . \]

\begin{prop} [Linearizing variables] \label{linearizingvariables} For $\pi \in S_n$, the maps
	\[ \tilde \phi_{\pi}^{\bO} : \cU^\bO_\pi \to \cD^{\pi} \times \fkl^0 \, , \ \   S \mapsto (D^\pi,Z),  \] \[ \tilde \phi_{\pi}^\bGL : \cU^\bGL_\pi \to \cD^{\pi} \times \fkl^0 \times \fkl^0 \, , \ \  M \mapsto (D^\pi, Y, Z)   \]
	are diffeomorphisms.
\end{prop}

\begin{proof} \medskip
	
	As the eigenvalues are simple, the gradients of the eigenvalues are nonzero, $\cD^{\pi_0}$ is transversal to $\cO^\bO$ and  $\cO^\bGL$ and the map is well defined: we consider its inverse. Assume that $D^\pi, Z $ and $Y $ are given, and search for $L  \in \bLone$ with $L  (Z + D^\pi ) \e (Y  + D^\pi ) L  $. Uniqueness and existence follow from the simplicity of the spectrum of $D^\pi $: entries of $L $ are computed along subdiagonals starting from the one closest to the diagonal, as in the proof of Proposition \ref{charts}. Now factor $L  = Q  R ,$ for $Q  \in \bOpos$, $R  \in \bUp$ and set $M = Q^\top (Y  + D^\pi ) Q $.
\end{proof}

\subsection{Toda flows in linearizing variables} \label{Todalinearized}

\bigskip
We obtain the push forward of a vector field in $\cU_{D^\pi}^\bGL \subset \cM$ to $\cD^{\pi} \times \fkl^0 \times \fkl^0 $. More precisely, for a variation $\dot M$ of a matrix $M \in \cU_{D^\pi}^\bGL$, we compute the  variations $\dot D^\pi, \dot Y$ and $\dot Z$ of the linearizing variables $(D^\pi, Y,Z)$ of Proposition \ref{linearizingvariables}.

Let $\pi \in S_n$ and $M \in \cU_{D^\pi}^\bGL$. From Proposition \ref{fatoracoes} (Schur), $M = Q^\top (Y + D^\pi) Q$, for $ Q \in \bOpos, \ Y \in \fkl^0$.

\begin{prop}[Splitting  the tangent space] \label{splitting} Let $M \in \cU_{D^\pi}^\bGL$ and $D^\pi, Q, Y$ as above. Then, for $\dot M \in \cM$, there is a unique decomposition
	\[  \dot M \e  \big( Q^\top \dot D^\pi Q\big) + \big( Q^\top \dot Y Q\big) + \big([ M ,  \ Q^\top \dot Q ]\big)   \, ,\]
	where the three summands belong to the  (real) vector spaces  (dependent of $M$),
	\[\ C =  \{ Q^\top E Q , \ E \in \cD\} , \  V \e  \{Q^\top N Q ,\ N \in \fkl0\} ,\  W \e  \{ [ M ,  A ]\, , \  A \in \fko \} ,   \]
	such that $\cM \e C \oplus V  \oplus W $.
	The maps
	\[ \dot M \in \cM \mapsto (\dot D^\pi\, ,  \dot Y \, ,  \ Q^\top \dot Q  ) \in \cD \times  \fkl^0  \times \fko  , \quad  B \in \fko \mapsto [Y+D^\pi,B] \in \tilde W\]
	are isomorphisms.
	\medskip
\end{prop}

\begin{proof} As $M = Q^\top (Y+D^\pi)Q$, $ Q \in \bOpos, \ Y \in \fkl^0$, taking derivatives,
	\begin{align*} \dot M &\e \dot Q ^\top (Y+D^\pi)Q + Q^\top \dot D^\pi Q + Q^\top \dot YQ  + Q^\top (Y+D^\pi)\dot Q \\
	&=  Q^\top \dot D^\pi Q + Q^\top \dot Y Q  + [ M , Q^\top \dot Q] 
	\end{align*}
As $Q^\top \dot Q \in \fko$, the three summands belong to $C, V$ and $W$ respectively. We  verify uniqueness of the splitting.	Write
	\[ W \e  \{ [ M ,  A ]\, , \  A \in \fko \} \e  \{ Q^\top [ Y+D^\pi ,  B ] Q\, , \  B \in \fko \}  \, . \]
	Getting rid of the conjugation by $Q$, we show that
	$\cM \e \tilde C \oplus \tilde V \oplus \tilde W$ for
	\[ \tilde C =  \cD \, , \quad \tilde V \e  \fkl_0 \, , \quad \tilde W \e  \{ [ Y+D^\pi ,  B ]\, , \ B \in \fko  \}.\]	
	Let $\tilde L \in \fkl^0$: we first show that for  an unknown $B \in \fko$, the equation
	\begin{equation}\label{transv}
		[Y + D^\pi , B] = \tilde L
	\end{equation}  implies $B=0$.  As in Proposition \ref{charts}, equate entries  along subdiagonals and use that $D^\pi$ has simple spectrum, to conclude that the strictly upper triangular entries of $B$ (and hence, its strictly lower triangular entries too) are equal to zero: $B$ is a diagonal matrix. Since $B \in \fko$, we must have $B=0$.
	
	In particular, the intersection of $\tilde V$ and $\tilde W$ is trivial. Similarly, $\tilde C$ and $\tilde W$ have trivial intersection and $B \in \fko \mapsto [Y+D^\pi,B] \in \tilde W$ is an isomorphism.
\end{proof}

Consider projections $\Pi_\fko$ and $\Pi_\fku$ associated with the direct sum $\cM \e \fko \oplus \fku$. Recall Equation \eqref{Toda} for the Toda flows associated with a polynomial $p$.

\begin{theo} [Toda flows linearize] \label{linearToda} The set $\cU_{D^\pi}^\bGL$ is invariant under Toda flows $\dot M = [M, \Pi_\fko p(M)]$. In linearizing variables $(D^\pi, Y, Z)$,
	\begin{equation}\label{evolucao} \dot D^\pi \e 0  , \quad  \dot Y = 0  ,  \quad \dot Z = [Z, - p(D^\pi)]  \, . 
		\end{equation}
	In particular, the solution $M(t)$ is globally defined: if $M(0) \in \cU_{D^\pi}^\bGL$, then, for all $t \in \RR$, we have $M(t) \in \cU_{D^\pi}^\bGL$. 
\end{theo}

\begin{remark} \label{formula}
		The equation for $Z$ is trivially solvable,
	as  $D(t)$ is a constant matrix, consisting of the spectrum of $M_0$ in some order:
	\[ Z(t) = \exp(tp(D^\pi))\ Z_0\, \exp(-tp(D^\pi)) \, , \;  z_{ij}(t) = z_{ij}(0) \, \exp (t(p(\lambda_{\pi(i)}) -p(\lambda_{\pi(j)} ))) . \]
Thus, the Toda flows {\it decouple} in such variables with multiplicative factors which  depend on the (constant) eigenvalues.
\end{remark}

\begin{proof} To simplify ideas, we consider $\pi = \pi_0$ and write $D^\pi = D$. Take $M \in \cU_{D^\pi}^\bGL$. Set $A = A(M) =  \Pi_\fko p(M) \in \fko$. As $\dot M = [M, A] \in W$, the unique splitting of $\dot M$ in the subspaces $C, V, W$  in Proposition \ref{splitting}, 
	\[  \dot M \e  \big( Q^\top \dot D^\pi Q\big)+ \big( Q^\top \dot Y Q\big)  + \big([ M , \Pi_{\fko} \ Q^\top \dot Q ]\big)  \, ,\]
	gives
	\[ Q^\top \dot E Q = 0 = \dot E \, , \quad  Q^\top \dot Y Q = 0 = \dot Y  \, , \quad Q^\top \dot Q = \Pi_{\fko} \ Q^\top \dot Q = A . \]
	We now compute $\dot Z$, for $Z + D = L^{-1}(Y+D)L$ by first computing $\dot L$. 
	Define projections $\Pi_{\fkl^0}$ and $\tilde \Pi_\fku$ associated with the direct sum $ \cM= \fklzero \oplus \fku$; the reader should not confuse $\tilde \Pi_\fku$ with $\Pi_\fku$, as in Equation \eqref{TodaJacobi}.
	
	As $Q = LU$
	for $Q \in \bOpos, L \in {\bf L^1}, U \in {\bf U}$, we have $\dot Q = \dot L U + L \dot U$
	and then 
	\[ L^{-1} \, \dot L \e  L^{-1} \,\dot Q \,U^{-1} - \dot U U^{-1} \e
	\Pi_\fklzero \, \big( L^{-1}\, \dot Q \, U^{-1} \big) \, . \]
	As $\dot Q = Q A$ and $U M U^{-1} = Z + D = L^{-1}(Y + D) L$ (see Equation \eqref{ZpD}), we have
	\begin{equation} \label{LL} \begin{aligned} 
		L^{-1} \, \dot L &\e \Pi_\fklzero \ \big( L^{-1}\, Q\, A  \ U^{-1} \big)    \e \Pi_\fklzero \, \big( U\, A  \ U^{-1} \big)  	\\
	 &\e \Pi_\fklzero \, \big( U\ (\Pi_\fko p(M))  \, U^{-1} \big)   \e \Pi_\fklzero \, \big( U\,  ( (p(M) -  \Pi_\fku \, p(M))  \ U^{-1} \big) \\
	&\e \Pi_\fklzero \, \big( U\  p(M) \,  \ U^{-1} \big) \e \Pi_\fklzero \,   p(U \ M\ U^{-1})
	\e \Pi_\fklzero \,   p(Z + D) \\
	&\e p(Z+D) - \diag p(Z+D)\e p(Z+D) - p(D) \, . \end{aligned} \end{equation}
	Taking derivatives of $Z + D = L^{-1}(Y + D) L$, one obtains, from $\dot Y = \dot D^\pi = 0$, \begin{equation} \label{Zdot}
		 \dot Z \e \dot{(Z + D)} \e [ Z + D, L^{-1} \dot L]\e [Z+D, p(Z+D) - p(D)] \\
	   \e [Z,  - p(D)] \, .
	\ \end{equation}
	The argument implies that solutions $M(t)$, in linearizing variables, are given by Equation \eqref{evolucao}. 	As  $p(D(t))$ is a constant diagonal matrix, $Z(t) \in \fkl$: the solution is globally defined  in   $ \cU_{D^{\pi_0}}^\bGL \subset \cU_{\pi_0}^\bGL$. Computations are similar for other permutations.  
	\end{proof}

\begin{remark} The theorem restricts to symmetric matrices in $\cU_{D^\pi}^\bO$: they are a special case when $Y=0$.
	\end{remark}

\begin{remark} The first formulae in Equation \eqref{Zdot} are similar to the original equation for a Toda flow. By the end of the equation, the equation for $\dot Z$ became a linear, decoupled, set of differential equations with constant coefficients.
\end{remark}

\begin{remark} The entries of $L$ also linearize, and for symmetric matrices $M$, decouple. Indeed, from Equation~(\ref{LL}),
\[
	L^{-1} \ \dot L \e   \   p(Z + D) - p(D) \e     L^{-1} p(Y + D) L - p(D) \, , \]
and thus, for constant matrices $Y$ and $D$ under Toda evolutions,
\[ \dot L \e p(Y+D)\, L - L\, p(D) \, .\]
In particular, by Proposition \ref{charts}, if $M$ is real symmetric, we have $Y = 0$ and $\dot L = [L, -p(D)]$. The matrices $Z$ however has an additional advantage as a coordinate system, which we explore in Subsection \ref{profiles}: they relate to profiles in a natural fashion.
\end{remark} 

\begin{remark} In numerical analysis, it is frequent to consider time dependent evolutions $\dot M = [M, \Pi_\fko p(M;t)]$, where the coefficients of $p$ depend on $t$ (a typical example is changing the shift parameter at each $QR$ iteration \cite{Parlett, LST2}).  The formulae in the statement of the theorem are easily adapted to handle this more general situation: in  linearizing variables $( D^\pi, Y, Z)$,
\[   \dot D^\pi \e 0  , \quad \dot Y = 0 \, , \quad \dot Z = [Z, - p(D^\pi;t)] \, \, , \]
\[ Z(t) = \exp\left(\int_0^t\  p(D^\pi;s) \ ds \right)\ Z_0\ \exp\left(-\int_0^t p(D^\pi;s) \ ds \right) \, . \]
\end{remark}

\bigskip
We use the above formulas to compute limits along an orbit. Notice that the result applies to nonsymmetric matrices, provided the initial condition has real, simple spectrum.
Recall that $D = D^{\pi_0} = \diag(\lambda_1, \ldots, \lambda_n)$, $\lambda_1 > \ldots> \lambda_n$.

\begin{prop} [Asymptotics]

		Consider an initial condition $M(0)\in \cU_{D^{\pi}}^\bGL$, $\pi \in S_n$. 
			Let $p: \RR \to \RR$ take different values at the eigenvalues of $D$  such that $p(\lambda_{\pi(1)}) > \ldots > p(\lambda_{\pi(n)})$. We then have
	\begin{equation} \label{asymptotics}
		 \lim_{t \to \infty} M(t) = (\phi_{\pi}^\bGL)^{-1} (D^{\pi}, Y(0), 0) , \quad \quad \phi_{\pi}^\bGL(M(0)) = (D^{\pi},  Y(0), Z(0)) \, .
		 \end{equation}
	 \end{prop}
 
 \begin{proof} From the expression for $Z(t)$ in Remark \ref{formula}, we have $Z(t) \to 0$, as $Z(t) \in \fkl$, and $Y(t), D(t)$ remain fixed. From Proposition \ref{charts}, $(D(0), Y(0), 0)$ corresponds to an upper triangular matrix $M(0)= U^{-1} (0 + D(0)) U$. 
 \end{proof}

\begin{remark} In particular, equilibria of the Toda flows are upper triangular (resp. diagonal) matrices for non-symmetric (resp. symmetric) initial conditions. A common scenario is $p$ being a strictly increasing function and $\pi = \pi_0$, as for the standard Toda flow $p(x) = x$. 
\end{remark}

It is an interesting fact that the asymptotics of such convergent orbits --- necessarily conserved quantities of the flow --- are naturally related to the variables $Y$. 
The asymptotic behavior of the standard Toda flow provides a projection $(D, Y, Z) \mapsto (D,  Y, 0)$ on the domain $\cU_{{\pi_0}}^\bGL$.

\subsection{Profiles} \label{profiles}

In Section \ref{example} we claimed that the isospectral sets of $3 \times 3$ tridiagonal symmetric matrices and the larger
isospectral sets of  upper Hessenberg matrices are manifolds to which the Toda flow is tangent,  and that the isospectral manifolds of tridiagonal symmetric matrices are covered by linearizing charts for the Toda flow.
We now extend the result for isospectral matrices with more general {\it profiles}  than the Jacobi and upper Hessenberg examples.

Given a set of pairs $S = \{ (i,j) , \ i \ge j, \  i, j \in \{1, 2, \ldots,n\} \}$, the {\it profile} $\p$ generated by $S$ is the set of pairs
\[ \p = \{ (i,j) \ | \ \exists \  (i', j') \in \p , \ i \le i', j \ge j'\} \, .\]
Let $V_\p^\cM \subset \cM$ be the subspace spanned by the matrices $E_{ij} = e_i \otimes e_j$ for $(i,j) \in \p$ and set $V_\p^\cS = V_\p^\cM \cap \cS$. Thus, for example,  the subspaces of real upper Hessenberg  and tridiagonal symmetric matrices equal the subspaces $V_\p^\cM$ and $V_\p^\cS$ for the profile $\p$ generated by the set $S = \{(2,1),(3,2), \ldots,(n,n-1)\}$.

An important algebraic property of the subspaces $V_\p^\cM$ is that they are modules for the adjoint action of the Lie algebra $\fku$. It follows from the expression $\dot M \e [ M , \Pi_\fko \,  p(M)]$ that Toda flows (\ref{Toda}) are tangent to any vector space of $\cM$ which is an $\fko$-module, for instance to the vector space $\cS$
of symmetric matrices. That is the reason why Toda flows are defined on the compact isospectral manifolds $\cO^\bO_D$.
But it also follows from the alternative expression $\dot M \e [ M , -\Pi_\fku \,  p(M)]$ that Toda flows are tangent to any subspace of $\cM$ which is an $\fku$-module, for instance the subspaces $V_\p^\cM$. That is, if a matrix belongs to $V_\p^\cM$ then its evolution by the Toda flow stays in $V_\p^\cM$. Thus for
any fixed simple spectrum and  profile the subsets $\cO^\bO_D \cap V_\p^\cS$ and $\cO_D^\bGL \cap V_\p^\cM$ are invariant by the Toda flow.

Unlike the case of $\cO^\bO_D$ and $\cO^\bGL_D$, which are orbits of group actions, there is
no quick argument to show that $\cO^\bO_D \cap V_\p^\cS$ and $\cO_D^\bGL \cap V_\p^\cM$ are manifolds. It is important to assume that
	the spectrum is simple: the isospectral set associated with real, $3 \times 3$ symmetric tridiagonal matrices with  eigenvalues $0$ and $1$, where 0 is of multiplicity 2, is a figure eight (two circles joined at a point). Symmetric tridiagonal matrices with fixed simple spectrum were shown to be a manifold in \cite{Tomei1}.
It is true that the orbits $\cO_D^\bO$ and $\cO_D^\bGL$  are transverse
to $V_\p^\cS$ and $V_\p^\cM$, but we prefer to check  transversality  working in the linearizing variables for the orbits. We  show that linearizing variables restrict well to the subspaces $V_\p^\cS$ and $V_\p^\cM$.

\begin{prop}  [Euclidean charts for $\cO^\bO_{D^\pi} \cap V_\p^\cS$ and $\cO_{D^\pi}^\bGL \cap V_\p^\cM$] \label{cartasprofiles} The sets $\cU_{D^\pi}^{\bO}\cap V_\p^\cS$ (resp. $\cU_{D^\pi}^\bGL \cap V_\p^\cM$) are open, dense subsets of $\cO^\bO_D\cap V_\p^\cS $ (resp. $\cO_D^\bGL \cap V_\p^\cM$). The restrictions
	\[ \phi_{D^\pi}^{\bO} : \cU_{D^\pi}^{\bO}\cap V_\p^\cS \to \fkl^0 \cap V_\p^\cS\, , \quad S \mapsto Z,  \]
	\[  \phi_{D^\pi}^\bGL : \cU_{D^\pi}^\bGL \cap V_\p^\cM\to \fkl^0 \times (\fkl^0\cap V_\p^\cM)   \, , \quad  M \mapsto ( Y, Z)   \]
	are  diffeomorphisms,  thus providing atlases for $\cO^\bO_{D^\pi} \cap V_\p^\cS$ and $\cO_{D^\pi}^\bGL\cap V_\p^\cM$ with Euclidean charts.
\end{prop}

\begin{proof} We prove the result for $\cO^\bO_D \cap V_\p^\cS$, the other being similar. Consider $\pi \in S_n$ for which $S \in \cU_{D^\pi}^{\bO}\cap V_\p^\cS$ and write, for the usual specifications,
	\[ S = Q^\top D Q = U^{-1} L^{-1} D L U = U^{-1} ( Z + D) U . \]
	Then $ U S U^{-1} = Z + D$. Since $S \in V_\p^\cM$ and $V_\p^\cM$ is invariant under conjugation by $\bU$, we must also have $U S U^{-1} = Z + D \in V_\p^\cM$, so that $Z \in 	 \fkl^0 \cap V_\p^\cS$. Therefore the restriction of $\phi_{D^\pi}^{\bO}$ sends $\cU_{D^\pi}^{\bO}\cap V_\p^\cS$ diffeomorphically onto  $\fkl^0 \cap V_\p^\cS$. Since $\fkl^0 \cap V_\p^\cS$ is a vector space we have constructed Euclidean charts for $\cO^\bO_D \cap V_\p^\cS$.
\end{proof}

Going back to the  $3 \times 3$ example, we consider upper Hessenberg matrices. The  chart domain 
$\cU_{D^\pi}^\bO \cap V_\p^\cS$ is diffeomorphic to $ \{(z_{21},z_{32} )\} \sim \RR^2$ and $\cU_{D^\pi}^\bGL \cap V_\p^\cM$ is diffeomorphic to
$\{( y_{21}, y_{31}, y_{32} ,z_{21}, z_{32})\} \sim \RR^5$.

\medskip
As an immediate consequence,  charts restrict well to profiles.

\begin{coro}  [Toda flows on $\cO^\bO_{D^\pi} \cap V_\p^\cS$ and $\cO_{D^\pi}^\bGL \cap V_\p^\cM$ linearize] \label{charts-profiles}  Symmetric and non-symmetric Toda flows  leave the sets $\cU_{D^\pi}^\bO\cap V_\p^\cS$ and $\cU_{D^\pi}^\bGL \cap V_\p^\cM$ invariant and linearize in variables   $(D^\pi, Z), Z\in \fkl\cap V_\p^\cM$ and
$(D^\pi, Y, Z), \ Y \in \fkl^0, \ Z \in \fkl^0\cap V_\p^\cM $.	
\end{coro}

\begin{proof}
The result follows from Theorem \ref{linearToda} and Proposition \ref{cartasprofiles}.
\end{proof}

\bigskip
Clearly linearizing variables (Proposition \ref{linearizingvariables}) also restrict to profiles.

\subsection{Toda flows as straight line motions, superintegrability} \label{lines}

We construct a  cover of the  manifolds $\cO^\bO \smallsetminus \cD$ and $\cO^\bGL \smallsetminus \fku$ by  open half-spaces with the following properties. Each half-space is invariant under any Toda flow. For a monotonic polynomial $p$, the associated Toda flow  becomes a straight line motion (with the same velocity along any line). 

Choose a permutation $\pi \in S_n$ and an index $(i_0, j_0), i_0 > j_0$. For a matrix $M \in \cU_{\pi}^\bGL \subset \cO^\bGL$, consider linearizing variables $(D^\pi,Y,Z)$. For a sign $\sigma \in \{+1, -1\}$, define the open set 
\[ \cH_{\pi, (i_0, j_0), \sigma}  \e \{ M \in \cU_{\pi}^\bGL \, | \; \sigma z_{i_0 j_0} > 0 \}  \ . \]
The disjoint union $\cH_{\pi, (i_0, j_0), +1} \cup \cH_{\pi, (i_0, j_0), -1} \subset \cU_{\pi}^\bGL$ is an open, dense subset.

Given a monotonic polynomial $p$, define new variables $w_{ij}, i > j ,$ in  $ \cH_{\pi, (i_0, j_0), \sigma}$:
\[ w_{i_0, j_0} \e \frac{ \log | z_{i_0j_0}|}{ p(\lambda_{i_0}) - p(\lambda_{j_0})} , \] 
\[ w_{i, j} \e \frac{z_{ij}}{| z_{i_0 j_0}|^{\varepsilon}} ,\quad\quad (i,j) \ne (i_0,j_0) \, , \quad \varepsilon = \frac{p(\lambda_i) - p (\lambda_j)}{p(\lambda_{i_0}) - p (\lambda_{j_0})} \ .\]
Coordinates $(D^\pi, Y, W)$ provide a diffeomorphism from $ \cH_{\pi, (i_0, j_0), \sigma}$ to $\cD^\pi \times \fkl \times \fkl$. Figure \ref{fig:H} gives a description of these variables.

\begin{figure}[ht]
	\def\svgwidth{75mm}
	\centerline{
\begingroup%
  \makeatletter%
  \providecommand\color[2][]{%
    \errmessage{(Inkscape) Color is used for the text in Inkscape, but the package 'color.sty' is not loaded}%
    \renewcommand\color[2][]{}%
  }%
  \providecommand\transparent[1]{%
    \errmessage{(Inkscape) Transparency is used (non-zero) for the text in Inkscape, but the package 'transparent.sty' is not loaded}%
    \renewcommand\transparent[1]{}%
  }%
  \providecommand\rotatebox[2]{#2}%
  \newcommand*\fsize{\dimexpr\f@size pt\relax}%
  \newcommand*\lineheight[1]{\fontsize{\fsize}{#1\fsize}\selectfont}%
  \ifx\svgwidth\undefined%
    \setlength{\unitlength}{557.62765528bp}%
    \ifx\svgscale\undefined%
      \relax%
    \else%
      \setlength{\unitlength}{\unitlength * \real{\svgscale}}%
    \fi%
  \else%
    \setlength{\unitlength}{\svgwidth}%
  \fi%
  \global\let\svgwidth\undefined%
  \global\let\svgscale\undefined%
  \makeatother%
  \begin{picture}(1,0.95846111)%
    \lineheight{1}%
    \setlength\tabcolsep{0pt}%
    \put(0,0){\includegraphics[width=\unitlength,page=1]{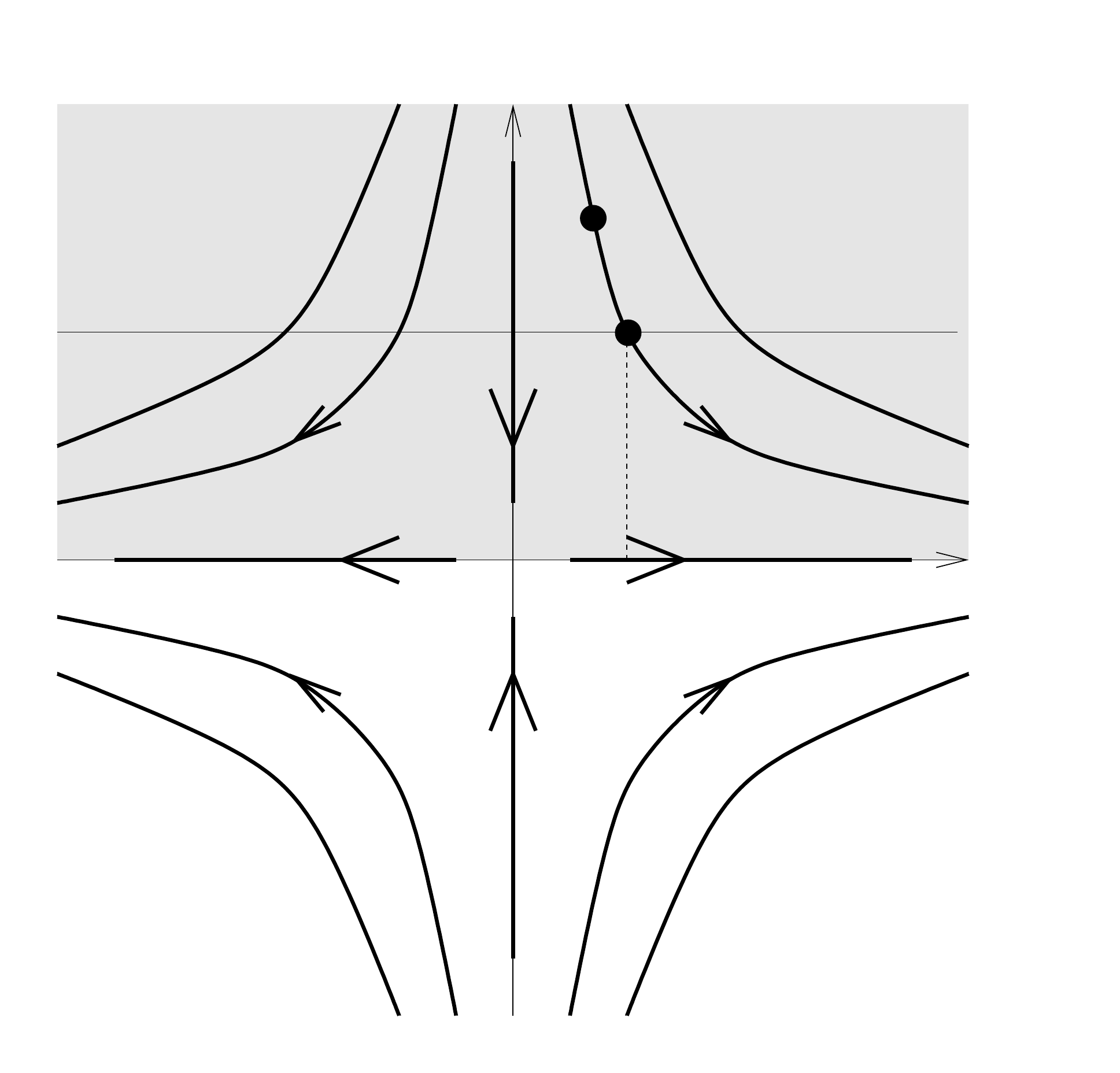}}%
    \put(0.43282167,0.89579134){\color[rgb]{0,0,0}\makebox(0,0)[lt]{\lineheight{1.25}\smash{\begin{tabular}[t]{l}$z_{i_0,j_0}$\end{tabular}}}}%
    \put(0.89070376,0.46334715){\color[rgb]{0,0,0}\makebox(0,0)[lt]{\lineheight{1.25}\smash{\begin{tabular}[t]{l}$z_{i,j}$\end{tabular}}}}%
    \put(0,0){\includegraphics[width=\unitlength,page=2]{H.pdf}}%
  \end{picture}%
\endgroup%
}
	\caption{The open set $\cU_\pi$ or, more precisely, a slice for fixed spectrum; we use $z_{ij}$ as coordinates.
		The half-space $\cH_{\pi, (i_0, j_0), +1}$ is represented by the upper half-plane. The vertical axis corresponds to $z_{i_0,j_0}$ and $w_{i_0,j_0}$ is a function of $z_{i_0, j_0}$. The horizontal axis corresponds to all other  $z_{i,j}$. The coordinate $w_{i,j}$ is obtained by following an orbit until it meets the horizontal hyperplane $z_{i_0, j_0} = 1$ and then reading the new value of $z_{i,j}$. }
	
	\label{fig:H}
\end{figure}

\begin{theo} [Straight line motion] Under the Toda flow associated with a monotonic polynomial $p$,
	\[ \dot D^\pi = 0 , \quad \dot Y = 0 , \quad \dot  w_{i_0, j_0} = 1 , \quad \dot  w_{i, j} = 0 , \, (i,j) \ne (i_0,j_0) \ . \]
	Thus such a flow is super-integrable in $ \cH_{\pi, (i_0, j_0), \sigma}$.
	\end{theo}

Super-integrability of a flow in a manifold of dimension $d$ means that there exist $d-1$ functionally independent conserved quantities.

\begin{proof}
The result is a simple application of Equations	\eqref{formula}.
\end{proof}

These variables restrict well to the symmetric case. Indeed, the open sets
\[  \cH^{\bO}_{\pi, (i_0, j_0), \sigma} = \cH_{\pi, (i_0, j_0), \sigma} \cap \cO^{\bO} \subset \cU_{\pi}^{\bO} \]
 define an open cover of $\cO^{\bO} \smallsetminus \cD$.
 For an index $(i_0,j_0)$, define $w_{ij}$ as above.
 Straight line motion works as above, except of course that we do not need $Y$. 
 
 In the same fashion, variables also restrict well to profiles: as in Corollary \ref{charts-profiles},  it suffices to discard the variables $(i,j)$ not in the profile.
 For instance, for $n \times n $ symmetric tridiagonal matrices, the linearizing variables are $\lambda_1, \ldots, \lambda_n$ and $z_{21}, z_{32}, \ldots, z_{n, n-1}$.
 We use a simple identification of Jacobi matrices in terms of linearizing variables already considered in \cite{LST1}.
 
 \begin{prop} \label{jacobi}
 	For any chart $\phi_{D^\pi}^\bO$,
 	Jacobi matrices in variables $(y_{ij}, z_{ij})$ correspond to entries $y_{ij}=0$, $z_{ij} =0$ for $j \ne i-1$ and $z_{i, i-1}>0$.
 \end{prop}
 
 \begin{proof} Let $\pi \in S_n$ and $D^\pi$ as usual. From the construction of $Z$ for tridiagonal symmetric matrices $T$, $y_{ij}=0$, $z_{ij} =0$  for $j \ne i-1$ and $z_{i, i-1} = 0$ if and only if $T_{i, i-1} = 0$. We are left with showing that the signs of $z_{i, i-1}$ and $T_{i, i-1}$ are equal. It suffices to show that this is true near $D^\pi= \diag(\lambda_{\pi(1)}, \ldots, \lambda_{\pi(n)})$. Disregarding profiles (as linearizing variables simply restrict to profiles, from Section \ref{profiles}), we consider matrices of the form $Q^\top D^\pi Q$ for $Q$ close to $I$, the identity matrix.  Denote by $\epsilon$ terms of lower order with respect to $\| Q - I\|$. For $Q = I + A + \epsilon , A\in \fko$,
 	\[   Q^\top D^\pi Q = ( I -  A + \epsilon) D^\pi (I + A + \epsilon) = D^\pi + [ D^\pi, A] + \epsilon \, . \]
 	Write $Q = LU, \in \bLone , U \in \bU$. Then 	
 	$L = I + \Pi_{\fkl^0} A + \epsilon$ and
 	\[ Z + D^\pi = L^{-1} D^\pi L = (I - \Pi_{\fkl^0} A) D^\pi (I + \Pi_{\fkl^0} A) = D^\pi + [ D^\pi , \Pi_{\fkl^0} \ A] + \epsilon \, . \]
 	The strictly lower triangular parts of $[ D^\pi, A]$  and $[ D^\pi , \Pi_{\fkl^0} \ A]$ are the same and, in particular, so are the signs of their entries.
 \end{proof}
 
  As a simple consequence, the half-space $ \cH^{\bO}_{\pi, (i_0, i_0-1), \sigma}$ is a set of matrices $T$ with $\sigma T_{(i_0, i_0-1)} > 0 $. The set of Jacobi matrices is the intersection of the half-spaces $ \cH^{\bO}_{\pi, (i_0, i_0-1), +1}$, for $i_0= 2, \ldots, n$. Figure \ref{fig:H} represents one such half-space for $n=3$. The first quadrant is the set of Jacobi matrices. More generally, Jacobi matrices would cover the first octant. In all cases, this octant is a Cartesian product in the variables $w_{i, i-1}$, implying super-integrability also in the set of Jacobi matrices for a monotonic polynomial $p$. 
  Super-integrability for the standard Toda flow (for which $p(x)=x$) on $n \times n $ Jacobi matrices  is due to Agrotis et al. \cite{Agrotis}.

\section{Linearizing SVD flows} \label{SVDlinearized}

Chu \cite{Chu} considered the differential equation on upper bidiagonal matrices $M$,
\begin{equation} \label{SVDflows}
 \dot M = M  \ \big( \Pi_\fko \, p(M^\top M) \big) - \big(\Pi_\fko p(M M^\top)\big) \  M \, ,
 \end{equation}
for  the special case $p(x) = x$. The projection  $\Pi_{\fko}$ is defined in Equation \eqref{TodaJacobi}.  Chu showed that $M(t) = \tilde Q^\top(t) M(0) \tilde U(t)$ for orthogonal matrices $\tilde Q(t),\tilde  U(t)$, and thus singular values of $M(t)$ and $M(0)$ are equal. Indeed, for the SVD representation $M(0)= Q_0^\top \Sigma U_0$, we obtain
\[ M(t) = \tilde Q^\top(t) \ Q_0^\top \Sigma U_0 \ \tilde U(t)  = Q^\top(t)\ \Sigma \ U(t) \, . \]
Chu also proved that, generically, as $t \to \infty$, we have $Q(t), U(t) \to I$.

Li \cite{Li} considered  smooth functions $f: \RR^+ \to \RR$ acting on (generic) non symmetric matrices. The time one map $M(0) \mapsto M(1)$ associated with the special case $p(x) = \log x$ is  the {\it zero-shift SVD algorithm} \cite{Golub} for computing the SVD representation of a  matrix. He then obtained the complete integrability of such equations for an open set of initial conditions in $\bGL$. A nontrivial task in his approach is the identification of an adequate symplectic structure.

\medskip
In this section we define linearizing variables for such flows.

\subsection{Manifolds of matrices with given singular values}

Let $\Sigma$ be a real diagonal matrix with simple, strictly positive spectrum. Define
\[ \cO^{\svd}_{\Sigma} \e  \{ Q^\top \ \Sigma\ U , \  \ Q, U \in \bO \} \, .\]
Standard algorithms to compute singular values are iterations within $\cO^{\svd}_{\Sigma}$ \cite{Golub} .

\bigskip
For $M \e Q^\top \Sigma U \in \cO^{\svd}_{\Sigma}, Q, U \in \bO$, write $M \e ( Q^\top \Sigma Q) (Q^\top U) \e S W$ for $S = Q^\top \Sigma Q \in \cO^\bO_\Sigma$ and $W = Q^\top U\in \bO$. The following result is the well known uniqueness of the polar decomposition  (see the Appendix).

\begin{prop} \label{uniqueSVD} The map $(S, W) \in \cO^\bO_\Sigma \times \bO \mapsto M \in \cO^{\svd}_{\Sigma} $ is a diffeomorphism.
\end{prop}


The result suggests an atlas for $\cO^{\svd}_{\Sigma}$, but we consider a more balanced construction using chart domains $\cU_{D^\pi}^{\bO}$ and $\cU_{D^{\rho}}^{\bO}$ (for permutations $\pi, \rho \in S_n$)  of the isospectral symmetric manifold $\cO^\bO_{\Sigma^2}$.

Recall that $\cD^{\sgn}$ is the group of real sign diagonal matrices.
For permutation matrices $P_\pi$ and $P_\rho$ and $E \in \cD^{\sgn}$, define
\[ \cU_{\pi,\rho,E} = \{ M \in \cO^{\svd}_{\Sigma} \ | \ M = Q^\top_\pi P^\top_\pi \ \Sigma E\  P_\rho U_\rho \, , \  Q_\pi,  U_\rho \in\bOpos\} \, . \]

\begin{prop} The sets $\cU_{\pi,\rho,E}$ cover $\cO^{\svd}_{\Sigma}$:  $\cO^{\svd}_{\Sigma} = \bigcup_{\pi, \rho, E} \ \cU_{\pi,\rho,E} $. The maps
	\[ \tilde \phi_{\pi, \rho,E} : \cU_{\pi,\rho,E} \to   \cU_{D^\pi}^{\bO} \times \cU_{D^{\rho}}^{\bO} , \ M \mapsto (S_\pi, S_\rho) = (Q^\top_\pi \Sigma^\pi Q_\pi, U^\top_\rho \Sigma^\rho U_\rho)\]
	are diffeomorphisms, as are $\phi_{\pi, \rho,E} : \cU_{\pi,\rho,E} \to \fkl^0 \times \fkl^0$, $M \mapsto (\phi_{D^\pi}^\bO(S_\pi), \phi_{D^\rho}^\bO(S_\rho))$.
\end{prop}

The maps $\phi_{D^\pi}^\bO$ and $\phi_{D^\rho}^\bO$ are the linearizing charts on $\cU_{D^\pi}^\bO$ and $\cU_{D^\rho}^\bO$ for the Toda flows defined in Proposition \ref{charts}.
Recall $\Sigma^\pi = \Pi_\pi^\top \Sigma \Pi_\pi$.

\begin{proof} Let $M \in \cO^{\svd}_{\Sigma}$. Then $M = \tilde Q^\top \Sigma \tilde U \, , \tilde Q, \tilde U \in \bO$. From the (PLU) decomposition (see  the Appendix), as in the proof of Proposition \ref{charts}, there are permutation matrices $P_\pi$ and $P_\rho$, acting on the columns of $\tilde Q^\top$ and on the rows of $\tilde U$, and $\tilde E_Q, \tilde E_P \in \cD^{\sgn}$ for which
	\[ \tilde Q^\top =   Q_\pi^\top P_\pi^\top \tilde E_Q   \, , \quad   \tilde U = \tilde E_U P_\rho   U_\rho , \quad \quad  Q_\pi, U_\rho \in  \bOpos \, . \]
	Thus  $M   =   Q_\pi^\top P_\pi^\top \tilde E_Q \ \Sigma \ \tilde E_U P_\rho   U_\rho=   Q_\pi^\top P_\pi^\top \tilde \Sigma \ E P_\rho   U_\rho $, for $E \in \cD^{\sgn}$, proving the covering property.	We now show that, for given $\pi, \rho$ and $E$ and $M \in \cU_{\pi,\rho,E}$, the decomposition 
	\begin{equation} \label{decspecial}
		 M = Q^\top_\pi P^\top_\pi \ \Sigma E\  P_\rho U_\rho \, , \quad \quad Q_\pi,  U_\rho \in\bOpos \, ,
		\end{equation}
	is unique. Write
	\[  M = \big( Q^\top_\pi P^\top_\pi \ \Sigma P_\pi Q_\pi\big) \big( Q_\pi^\top P_\pi^\top E\  P_\rho U_\rho \big)  = \big( Q^\top_\pi  \Sigma ^\pi Q_\pi\big) \big( Q_\pi^\top P_\pi^\top E\  P_\rho U_\rho \big) = S_\pi W_\pi . \]
	Thus another decomposition $ M = \hat Q^\top_\pi P^\top_\pi \ \Sigma  E\  P_\rho \hat U_\rho$ for $ \hat Q_\pi,  \hat U_\rho \in\bOpos $ yields $M = \hat S_\pi \hat W_\pi$ and, from Proposition \ref{uniqueSVD}, $S_\pi = \hat S_\pi,  W_\pi = \hat  W_\pi $. Explicitly,
	\[ Q^\top_\pi  \Sigma ^\pi Q_\pi = \hat Q^\top_\pi  \Sigma ^\pi \hat Q_\pi \, , \quad Q_\pi^\top P_\pi^\top E\  P_\rho U_\rho =\hat Q_\pi^\top P_\pi^\top  E\  P_\rho \hat U_\rho \, .\]
	From the first equality, since $Q_\pi, \hat Q_\pi \in \bOpos$, we have $Q_\pi = \hat Q_\pi$. From the second, $ U_\rho =  \hat U_\rho$. Thus, the decomposition (\ref{decspecial}) is indeed unique.
	
	For $M \in \cU_{\pi,\rho,E}$, the decomposition obtains all the matrices required for the definition of the chart $\tilde \phi_{\pi, \rho,E}$: it is well defined and smooth. Surjetivity is immediate. Injectivity follows from the uniqueness of the decomposition (\ref{decspecial}).
\end{proof}

\subsection{SVD flows in linearizing variables} \label{SVDlinearized2}

For polynomials $p$ and $q$, an extension of the differential Equation~(\ref{SVDflows}) is
\begin{equation} \label{SVDflowsgen}
	\dot M = M  \ \big( \Pi_\fko \, p(M^\top M) \big) - \big(\Pi_\fko\ q(M M^\top)\big) \  M \, .
\end{equation}
For $M \in \cU_{\pi,\rho,E} $, as $ M = Q^\top_\pi P^\top_\pi \ \Sigma E\  P_\rho U_\rho$, we  have
\[ W_\pi = M M^\top = Q_\pi^\top (\Sigma^\pi)^2 Q_\pi = S_\pi^2  \, , \quad  W_\rho = M^\top M = U_\rho^\top (\Sigma^\rho)^2 U_\rho= S_\rho^2  \]
and then Equation~(\ref{SVDflowsgen}) leads to two decoupled Toda flows on $\cO_{\Sigma^2}^\bO$,
\[ \dot W_\pi = [ W_\pi, \Pi_\fko\ q(W_\pi)] \, , \quad \dot W_\rho = [  W_\rho, \Pi_\fko\ p(W_\rho)]\, ,\]
which are linearizable from Section \ref{Todalinearized}. Embedding the SVD flows into a collection of  straight line motions now mimics the construction in Subsection \ref{lines}. We do not consider issues related to profiles.

\section{Extended calculi} \label{calculi}

We introduce two classes of flows  which are explicitly solved in linearizing coordinates. The first includes the Toda flows and admits $QR$ steps. The relevance of such flows to numerics is far from being established.

For  $C \in \fkl$, define  the {\it extended calculi}
\[ \cC_Q(M, C) = Q^\top C Q \, , \quad  \cC_L(M, C) = L^{-1} C L \, , \quad\cC_U(M, C) = U^{-1} C U  \, . \]
Here, we assume $M \in \cU_{D^\pi}^\bGL$, so that the  decompositions
\[ M = Q^\top (Y+D) Q \, , \quad Z+ D = L^{-1} (Y+D) L \, , \quad M = U^{-1} (Z+D) U\] define $Q, L$ and $U$. We drop the dependence on $\pi$ in $D^\pi$ for convenience. 
The definitions  extend the standard functional calculus:
\[ p(M) = \cC_Q(M, p(Y+D)) = \cC_U(M, p(Z+D)) \, , \quad p(Z+D) = \cC_L(Z, p(Y+D))  \, . \]
The Toda flows in this notation becomes
\[ \dot M \e [ M , \Pi_\fko \, \cC_Q(M, p(Y+D))] \e [ M , \Pi_\fko \, \cC_U(M, p(Z+D))] \ . \]
and we consider a natural larger class.

\begin{prop} [Some explicitly solvable flows] Let $C \in \fkl$ and $M_0 \in \cU_{D^\pi}^\bGL$ be constant matrices. The solution $M(t)$ of
\begin{equation} \label{novosfluxos} \dot M  = [M , \Pi_\fko \, \cC_Q(M, C)] \, , \quad M(0) = M_0
	\end{equation}
	is globally defined in $\cU_{D^\pi}^\bGL$. In linearizing coordinates,  we have
	 $D(t) = D(0)$, $Y(t) = Y(0)$ and 
	\[ Z(t) + D  = L^{-1}(t) (Y + D)(0) L(t) \, , \quad \quad L(t) = \exp(tC)\ L(0) \exp(-t \diag C)\, . \]
\end{prop}

\begin{proof} Following the proof of Theorem \ref{linearToda}, $\dot Y = \dot D = 0$ and, from Equation~\eqref{LL},
	\begin{align*}
	L^{-1} \ \dot L   &\e \Pi_\fklzero \ \big( U \ ( \Pi_\fko \, \cC_Q(M,C) ) U^{-1} \big) 
	 \e  \Pi_\fklzero \ \big( U \   \cC_Q(M,C)  U^{-1} \big) \\ &\e
	\Pi_\fklzero \ \big( U \   Q^\top C Q  U^{-1} \big) \e 	\Pi_\fklzero \ \big( L^{-1} C L \big) \\ &\e L^{-1} C L - \diag (L^{-1} C L) \e L^{-1} C L - \diag C . 
	\end{align*}
	Thus $\dot L \e L (L^{-1} C  L - \diag C) \e C L - L \diag C$, implying the formula for $L(t)$ in the statement. The other claims are easy. 
\end{proof}

\begin{remark} It is easy to obtain   $Q^\top \dot Q = \Pi_\fklzero Q^\top C Q$. In particular, such equation is also explicitly solvable, from the Schur decomposition of $M(t)$.
	\end{remark}

Here is another collection of flows easily described in linearizing variables.

\begin{prop}  For  $C \in \fkl$ and $M_0 \in \cU_{D^\pi}^\bGL$, the solution of
	\begin{equation} \label{Yfacil}
		\dot M = [ M,\cC_Q(M, C)] \, , \quad M(0) = M_0 \, .
	\end{equation}
is globally defined in $\cU_{D^\pi}^\bGL$. For the linearizing variables  $(D(t), Y(t), Z(t))$,
\[\dot D^\pi = 0 \, , \quad \dot Y = [ Y  + D , C] \, , \quad \dot Z = [ Z + D ,  \, \cC_L(M, C) ] \, . \]
\end{prop}

\begin{remark} The solution of the differential equation for $Y(t)$ is
\[ Y(t) = \exp(-tC) \ \big( Y_0 + \int_0^t \exp(- s C) [  D , C] \ \exp( s C) \ ds \ \big) \ \exp(tC) \, .\]
For $C = E_{ij} = e_i \otimes e_j = e_i e_j^\top$,
$ Y(t) = Y_0 + t [Y_0 + D, E_{ij}] - t^2  (Y_0)_{ij} E_{ij}$.
For $C = E_{ii}$, $Y(t) = \big( I + (e^{-t} -1) E_{ii} \big) \  Y_0 \ \big( I + (e^{t} -1) E_{ii} \big) $ .
\end{remark}

\begin{proof} Since the differential equation is a Lax pair,  $\dot D = 0$. As usual, write $M = Q^\top(Y + D) Q$. 
	Now differentiate to get
	\begin{align*}
	 \dot M &= \dot Q^\top(Y + D) Q + Q^\top \ \dot Y \  Q + Q^\top(Y + D) \dot Q \\
	&= Q^\top \ (Q \dot Q^\top) (Y+D)\ Q + Q^\top \ \dot Y \ Q + Q^\top\ (Y+D) \ (\dot Q Q^\top) \ Q \\ &=  Q^\top \left(  [(Y+D) , \dot Q Q^\top] + \dot Y \right) \, Q 
	\end{align*}
Write Equation~(\ref{Yfacil}) as $\dot M = Q^\top [ Y+D , C] Q$ 
	to obtain
	\[ [ Y+D , C] = [ Y+D , \dot Q Q^\top] + \dot Y \, . \]
	Write $\hat C = C - \diag C = \Pi_\fklzero \, C$, so that
	\[ [ Y+D , \hat C] + [ Y+D ,\diag C] = [ Y+D , \dot Q Q^\top] + \dot Y \, . \]
	Let $\tilde V = \fkl^0$ and $\tilde W$ be the subspaces defined in the proof of Proposition \ref{splitting}. We have  $[ Y+D ,\diag C] = [ Y ,\diag C] \in \fkl^0$, $[ Y+D , \hat C], \dot Y \in \fkl^0 $ and therefore  $[ Y+D , \dot Q Q^\top] \in \fkl^0 = \tilde V$. We also have $[ Y+D , \dot Q Q^\top] \in  \tilde W$ and therefore, from Proposition \ref{splitting}, we must have
	\[ [ Y+D , \dot Q Q^\top] = 0 \, , \quad  \dot Y = [ Y+D , C]  \, . \]
	From Equation~(\ref{transv}), $\dot Q Q^\top = 0$, so that $\dot Q = 0$. In particular, for 
	$Q = LU$, we must have that $L$ and $U$ are constant. To obtain the evolution of $Z$, differentiate $Z+ D = L^{-1}(Y+D) L$ and use that $\dot L = 0$: 
	\[ \dot Z = [ Z+D, L^{-1} \dot L] + L^{-1} \dot Y L ] = 0 + L^{-1} [ Y + D, C] L = [Z+D, L^{-1} C L] \, , \]
	completing the proof.
	\end{proof}

%
%
%

\bigskip

 In numerical analysis, Toda flows arise as interpolations of matrix iterations preserving spectrum. Practical algorithms are described by iterations.

The archetipical example is the $QR$-step. Start with an invertible matrix $M = M_0$, write its $QR$ decomposition $M_0 = Q_0 R_0$ and set
\[ M_1 = R_0 Q_0 = Q_0^\top M_0 Q_0 = R_0 M_0 R_0^{-1} \, . \]

A simple approach is to search for a solution of the flow in terms of a {\it  Semenov-Tian-Shanski formula} \cite{STS} (also \cite{Symes, Symes2}). By this we mean the following. For a constant matrix $X$, consider the $QR$ decomposition
\[ e^{t X} = [e^{t X}]_Q [e^{t X}]_R = Q(t) \ R(t) \]
and verify if the conjugation $M(t) = Q^\top(t) M(0) Q(t)$ defines a bona fide flow.
The special case $X = \log M_0$, for example, gives rise to the standard Toda flow, $\dot M = [M , \Pi_\fko \, M] , M(0) = M_0$. Indeed, from the formula,
\[ M(1) = [e^{\log X}]_Q [e^{\log X}]_R = [M_0]_Q [M_0]_R = Q_0 R_0 = M_1\, . \]
Another interesting example is given in Theorem 7 of \cite{DLT2}. Notice that $X$ may not commute with $M_0$ and thus it is possible that conjugations by $[e^{t X}]_Q$ (as we consider) or by $[e^{t X}]_R$ may lead to different flows.

\begin{prop} Equation (\ref{novosfluxos}) is also solved by
	\[ M(t) = Q_e(t)^\top M_0 Q_e(t) \, , \ \  \hbox{where}\ \  e^{t\cC_Q(M_0, C)} = Q_e(t) R_e (t) , \ Q_e \in \bO \, , R_e \in \fku \, .  \]
\end{prop}

\begin{proof} Let $M_e(t)= Q_e(t)^\top M_0 Q_e(t)$: we must show that $M_e (t)$ satisfies Equation~(\ref{novosfluxos}). Taking derivatives of $M_e (t)$, we obtain $\dot M_e = [ M_e , Q_e^\top \dot Q_e] $. Taking derivatives of the decomposition yielding $Q_e R_e$ (as in the proof of Theorem \ref{linearToda}), we obtain $Q_e^\top \dot Q_e = \Pi_\fko \, Q_e^\top C_Q(M_0, C) Q_e$. Adding up,
	\[ M_e  = [ M_e , \Pi_\fko \, Q_e^\top C_Q(M_0, C) Q_e]  = [ M_e , \Pi_\fko \, C_Q(M_e, C)] \, , \quad M_e(0)= M_0 \, , \]
	which is Equation~(\ref{novosfluxos}). By uniqueness, $M(t) = M_e(t)$.
\end{proof}

\section{Appendix: some basic linear algebra} \label{appendix}

The following facts from linear algebra are well known (\cite{HornJohnson}).

\medskip
\noindent (PLU) For $X \in \bGL$, there are $\pi \in S_n$, $  L \in \bLone , U \in \bU $ for which $ X = P_\pi L U $.

\medskip
The standard matrix decompositions  below are unique and smooth.

\medskip

\noindent	(LU) For $X \in \bGLpos$, there are $  L \in \bLone , U \in \bUp $ such that $X = L U $.

\smallskip
\noindent (QR) For $X \in \bGL$, there are $Q \in \bO$, $R \in \bUp$ such that $X = QR$.

\smallskip
\noindent(Polar) For $M \in \bGL$, there are $P$ positive symmetric, $Q \in \bO$ such that $M = PQ$.

\medskip

The Schur decomposition admits a similar result (\cite{HornJohnson}), but we only present in the proposition below the specific format fitting our purposes, together with a minor modification of the QR factorization. The sets  $ \cU_{D^\pi}^{\bO}$ and $ \cU_{\pi}^{\bGL}$  are defined in Equation~\eqref{cartas}.

\medskip

\begin{prop} \label{fatoracoes}
The following factorizations are unique and smooth.

\medskip
\noindent (QR) For $X \in \bGLpos$,   $X = QR$ where $Q \in \bOpos$ and $R \in \bUp$.  Similarly, $X = LQ$, for $ L \in \bLp$, and $Q \in \bOpos$. 

\smallskip
\noindent (Schur) Let $ \pi \in S_n$ and $D$ be a real diagonal matrix with simple spectrum. Set $D^\pi = \pi^\top D \pi$. For $S \in \cU_{D^\pi}^{\bO}$,  $S = Q^\top D^\pi Q$ for $Q \in \bOpos$.
For   $M \in \cU_{D^\pi}^\bGL$,  $M = Q^\top ( Y + D^\pi) Q$ for $Q \in \bOpos, Y \in \fkl^0$.

\end{prop}

\begin{proof}
We sketch  existence and uniqueness for the reader's convenience.

 For the QR factorization  $X=LQ$, apply the Gram-Schmidt algorithm to the rows of $X$.
 To show uniqueness, suppose
$ X = L_1 Q_1  = L_2 Q_2 $ so that $Q_2 Q_1^{-1} = L_2^{-1} L_1$. The left (resp. right) hand side is orthogonal (resp. lower triangular): they must be equal to a  diagonal  matrix $E$ with diagonal entries of modulus one. As $Q_1, Q_2 \in \bOpos$, we  have $E = I$, the identity matrix, and uniqueness follows. The decomposition $X = QR$ is treated similarly.

$M$ is real, symmetric, the Schur triangularization of $M$ reduces to the spectral theorem. In general, for $M \in \cU_{D^\pi}^\bGL$, we have $M = P^{-1} D^\pi P$ with $P \in GL^p$, and the QR factorization gives $P = L Q$ for $ L \in \bL$, and $Q \in O^s$. Thus $M = Q^\top L^{-1} D^\pi L Q = Q^\top (Y + D^\pi)Q$ for $Y \in \fkl^0$. To prove uniqueness, let $M = Q_1^\top (Y_1 + D^\pi) Q_1 = Q_2^\top (Y_2 + D^\pi) Q_2$ and, for $Q = Q_2 Q_1^\top$,
\[ Q (Y_1 + D^\pi) = (Y_2 + D^\pi) Q \, .\]
Equating entry $(1,n)$ and using the simplicity of the spectrum of $D^\pi$, one learns that $Q_{1,n}=0$. Equate entries $(1, n-1)$ and $(2,n)$ to learn that the corresponding entries of $Q$ are also zero. Proceed along subdiagonals, to conclude that $Q$ is diagonal. As $Q_1, Q_2 \in \bOpos$, one must have $Q_1 = Q_2$,	and thus $Y_1 = Y_2$. If $M=S$ is symmetric, we must have $L=0$.
\end{proof}

\bigskip\bigskip\bigbreak

{
	
	\parindent=0pt
	\parskip=0pt
	\obeylines
	
	Ricardo S. Leite, Departamento de Matem\'atica, UFES
	Av. Fernando Ferrari, 514, Vit\'oria, ES 29075-910, Brazil
	
	\smallskip
	
	Nicolau C. Saldanha and Carlos Tomei, Departamento de Matem\'atica, PUC-Rio
	R. Marqu\^es de S. Vicente 225, Rio de Janeiro, RJ 22451-041, Brazil
	
	\smallskip
	
	David Mart\'inez Torres, Departamento de Matem\'atica Aplicada, Secci\'on de Arquitectura, Universidad Polit\'ecnica de Madrid, Avda. Juan de Herrera 4, Madrid  28040, Spain

	\smallskip
	
	ricardo.leite@ufes.br
	saldanha@puc-rio.br
	carlos.tomei@puc-rio.br
	dfmtorres@gmail.com
	
}

\end{document}